\documentclass[a4paper]{article}

\usepackage{amsmath,amsfonts,amssymb,amsthm}
\usepackage{graphicx}
\usepackage{color}
\usepackage[all]{xy} 
\usepackage{hyperref} 
\usepackage{tikz}
\usetikzlibrary{arrows}

\DeclareMathOperator{\supp}{supp}
\DeclareMathOperator{\sign}{sign}
\DeclareMathOperator{\spann}{span}

\newcommand{\RR}{{\mathbb R}}

\newtheorem{thm}{Theorem}
\newtheorem{pro}[thm]{Proposition}
\newtheorem{lem}[thm]{Lemma}

\newtheorem{dfn}[thm]{Definition}
\newtheorem{exa}{Example}

\makeatletter
\def\blfootnote{\xdef\@thefnmark{}\@footnotetext}
\makeatother

\begin{document}

\title{Elementary vectors and conformal sums \\ in polyhedral geometry
and their relevance \\ for metabolic pathway analysis}

\author{Stefan M\"uller$^*$ and Georg Regensburger}

\date{\today}

\maketitle

\begin{abstract}
\noindent
A fundamental result in metabolic pathway analysis states
that every flux mode can be decomposed into a sum of elementary modes.
However, only a decomposition without cancelations is biochemically meaningful,
since a reversible reaction cannot have different directions in the contributing elementary modes.
This essential requirement has been largely overlooked by the metabolic pathway community.

Indeed,
every flux mode can be decomposed into elementary modes without cancelations.
The result is an immediate consequence of a theorem by Rockafellar
which states that every element of a linear subspace is a conformal sum (a sum without cancelations)
of elementary vectors (support-minimal vectors).
In this work, we extend the theorem, first to ``subspace cones''
and then to general polyhedral cones and polyhedra.
Thereby, we refine Minkowski's and Cara\-th\'eo\-dory's theorems,
two fundamental results in polyhedral geometry.
We note that, in general, elementary vectors need not be support-minimal;
in fact, they are conformally non-decomposable and form a unique minimal set of conformal generators.

Our treatment is mathematically rigorous, but suitable for systems biologists,
since we give self-contained proofs for our results
and use concepts motivated by metabolic pathway analysis.
In particular, we study cones defined by linear subspaces and nonnegativity conditions -- like the flux cone --
and use them to analyze general polyhedral cones and polyhedra.

Finally,
we review applications of elementary vectors and conformal sums in metabolic pathway analysis.

\end{abstract}

\blfootnote{
\scriptsize

\noindent
{\bf S.\ M\"uller} (\href{mailto:stefan.mueller@ricam.oeaw.ac.at}{stefan.mueller@ricam.oeaw.ac.at}).  
Johann Radon Institute for Computational and Applied Mathematics,
Austrian Academy of Sciences,
Apostelgasse 23, 1030 Wien, Austria
\smallskip

\noindent
{\bf G.\ Regensburger} (\href{mailto:georg.regensburger@ricam.oeaw.ac.at}{georg.regensburger@ricam.oeaw.ac.at}).  
Johann Radon Institute for Computational and Applied Mathematics,
Austrian Academy of Sciences,
Altenbergerstra{\ss}e 69, 4040 Linz, Austria
\smallskip

\noindent
$^*$Corresponding author
}


\section{Introduction}

Cellular metabolism is the set of biochemical reactions
which transform nutrients from the environment into all the biomolecules a living cell consists of.
Most metabolic reactions are catalyzed by enzymes,
the expression and activity of which is controlled by gene and allosteric regulation, respectively.

A metabolic network together with enzymatic reaction rates gives rise to a nonlinear dynamical system
for the metabolite concentrations.
However,
for genome-scale networks,
quantitative knowledge of the underlying kinetics is not available,
and a mathematical analysis is not practicable.
Instead,
one considers only stoichiometric information
and studies the system of linear equalities and inequalities for the fluxes (net reaction rates),
arising from the pseudo steady-state assumption
and irreversibility constraints.

A metabolic network is given by $n$ internal metabolites, $r$ reactions,
and the corresponding stoichiometric matrix $N \in \RR^{n \times r}$,
which contains the net stoichiometric coefficients of each metabolite in each reaction.
The set of irreversible reactions is given by $\mathcal{I} \subseteq \{1,\ldots,r\}$.
One is interested in the {\em flux cone}
\[
C = \{ f \in \RR^r \mid N f = 0 \text{ and } f_i \ge 0 \text{ for } i \in \mathcal{I} \} ,
\]
which is a polyhedral cone defined by the null-space of the stoichiometric matrix and nonnegativity conditions.
Its elements are called {\em flux modes}.

As a running example,
we consider a small network, taken from~\cite{SchusterHilgetagWoodsFell2002},
the corresponding stoichiometric matrix,
and the resulting flux cone:

\bigskip
\hspace{-0.05\textwidth}
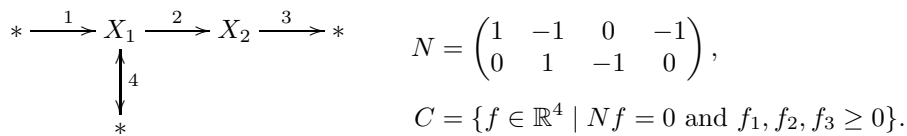
\begin{minipage}[c]{0.40\textwidth}
\xymatrix{
\ast \ar[r]^1 & X_1 \ar[r]^2 \ar@{<->}[d]^4 & X_2 \ar[r]^3 & \ast \\
& \ast
}
\end{minipage}
\hfill
\begin{minipage}[c]{0.55\textwidth}
\vspace{-2.ex}
\begin{align*}
N &=
\begin{pmatrix}
1 & -1 & 0 & -1 \\
0 & 1 & -1 & 0
\end{pmatrix}  , \\[1ex]
C &= \{ f \in \RR^4 \mid N f = 0 \text{ and } f_1,f_2,f_3 \ge 0 \} .
\end{align*}
\end{minipage}
\bigskip

\noindent
The network consists of two internal metabolites $X_1,X_2$ and four chemical reactions.
Reaction~1 imports $X_1$ from the environment (indicated by the symbol~$\ast$) which yields the first column $(1,0)^T$ of the stoichiometric matrix $N$.
Reaction~2 transforms $X_1$ into $X_2$ which gives the column $(-1,1)^T$, and reaction~3 exports $X_2$ which gives $(0,-1)^T$.
The first three reactions are assumed to be irreversible which yields the nonnegativity constraints $f_1,f_2,f_3 \ge 0$ in the definition of the flux cone $C$.
Finally, reaction~4 is reversible and exports/imports $X_1$.

Metabolic pathway analysis aims to identify biochemically/biologically/bio\-techno\-logically meaningful routes in a network,
in particular, the smallest routes.
Several definitions for minimal metabolic pathways have been given in the literature,
with {\em elementary modes} (EMs) being the fundamental concept
both biologically and mathematically~\cite{KlamtStelling2003,LlanerasPico2010}.
Formally,
EMs are defined as support-minimal
(or, equivalently, support-wise non-decomposable)
flux modes~\cite{SchusterHilgetagWoodsFell2002,SchusterHilgetag1994}.
Clearly, a positive multiple of an EM is also an EM
since it fulfills the steady-state condition and the irreversibility constraints.

In the example, the EMs are given by $e^1=(1,0,0,1)^T$, $e^2=(0,1,1,-1)^T$, $e^3=(1,1,1,0)^T$, and their positive multiples.
It is easy to check that $e^1$, $e^2$, and $e^3$ are flux modes (elements of the flux cone) and support-minimal.
Note that $e^3=e^1+e^2$.

A fundamental result in metabolic pathway analysis states
that every flux mode can be decomposed into a sum of EMs~\cite{SchusterHilgetagWoodsFell2002}.
However,
only a decomposition without cancelations is biochemically meaningful,
since a reversible reaction cannot have different directions in the contributing EMs.
This essential requirement has been largely overlooked by the metabolic pathway community.
Indeed,
as we will show in this work,
every flux mode can be decomposed into EMs {\em without cancelations},
that is,
\begin{itemize}
\item[$(0)$] if a component of the flux mode is zero,
then this component is zero in the contributing EMs,
\item[$(+)$] if a component of the flux mode is positive,
then this component is positive or zero in the contributing EMs,
\item[$(-)$] if a component of the flux mode is negative,
then this component is negative or zero in the contributing EMs.
\end{itemize}
In mathematical terms,
every nonzero element of a "subspace cone" (defined by a linear subspace and nonnegativity conditions) is a conformal sum of elementary vectors,
cf.\ Theorem~\ref{thm}.
The result is stated in~\cite{UrbanczikWagner2005, Urbanczik2007};
part $(0)$ has been shown in~\cite{SchusterHilgetagWoodsFell2002}
and guarantees a decomposition without cancelations in a weaker sense~\cite{LlanerasPico2010,Zanghellini2013}.

In the example, the flux mode $f = (2,1,1,1)^T$ can be decomposed into EMs in two ways:
\begin{align*}
f =
\begin{pmatrix}
2 \\ 1 \\ 1 \\ 1
\end{pmatrix}
&= 2 \, e^1 + e^2 =
\begin{pmatrix}
2 \\ 0 \\ 0 \\ 2
\end{pmatrix}
+
\begin{pmatrix}
0 \\ 1 \\ 1 \\ -1
\end{pmatrix}
\\
&= e^1 + e^3 =
\begin{pmatrix}
1 \\ 0 \\ 0 \\ 1
\end{pmatrix}
+
\begin{pmatrix}
1 \\ 1 \\ 1 \\ 0
\end{pmatrix} .
\end{align*}
The first sum involves a cancelation in the last component of the flux.
The last reaction is reversible, however, it cannot have a net rate in different directions at the same time.
Hence, only the second sum is biochemically meaningful.
As stated above, a decomposition without cancelations is always possible.

In convex analysis,
elementary vectors of a linear subspace were introduced as support-minimal vectors by Rockafellar in 1969.
He proves that every vector is a conformal sum (originally called harmonious superposition)
of elementary vectors~\cite[Theorem~1]{Rockafellar1969}.
For proofs and generalizations in the settings of polyhedral geometry and oriented matroids,
see \cite[Lemma 6.7]{Ziegler1995} and \cite[Theorem 5.36]{BachemKern1992}.
Rockafellar points out that this result is easily shown to be equivalent to Minkowski's theorem~\cite{Minkowski1896}
for pointed polyhedral cones,
stating that every nonzero vector is a nonnegative linear combination of extreme vectors. 
Moreover, the result immediately implies Carath\'eodory's theorem \cite{Caratheodory1911},
stating that the number of extreme vectors in such a nonnegative linear combination
need not exceed the dimension of the cone.
In fact, Rocka\-fellar writes:
``This is even a convenient route for attaining various important facts about polyhedral convex cones,
since the direct proof 
[...]\ for Theorem 1 is so elementary.''

In metabolic pathway analysis,
decompositions without cancelations 
were introduced by Urbanczik and Wagner~\cite{UrbanczikWagner2005}.
The corresponding elementary vectors are defined by intersecting a polyhedral cone with all closed orthants of maximal dimension.
By applying Minkowski's theorem for pointed polyhedral cones,
every vector is a sum of extreme vectors without cancelations.
Urbanczik further extended this approach to polyhedra arising from flux cones and inhomogeneous constraints~\cite{Urbanczik2007}.

In polyhedral geometry,
it seems that conformal decompositions of general cones and polyhedra have not yet been studied.
In this work, following Rocka\-fellar,
we first extend his result to cones defined by linear subspaces and nonnegativity conditions (Theorem~\ref{thm}).
For subspace cones, support-minimality is equivalent to conformal non-decomposability.
As it turns out, for general polyhedral cones, elementary vectors have to be defined as conformally non-decomposable vectors.
However, these are in one-to-one correspondence with elementary vectors of a higher-dimensional subspace cone,
and, by our result for subspace cones,
we obtain a conformal refinement of Minkowski's and Carath\'eodory's theorems
for polyhedral cones (Theorem~\ref{thm:polcon}).
In particular, there is an upper bound on the number of elementary vectors
needed in a conformal decomposition of a vector.
Finally, by taking into account vertices and conformal convex combinations,
we further extend our result to polyhedra (Theorem~\ref{thm:pol}).
We note that elementary vectors do not form a minimal generating set (of an s-cone, a general polyhedral cone, or a polyhedron).
However, they form a unique minimal set of {\em conformal} generators (Proposition~\ref{pro:min}).


\section{Definitions}

We denote the nonnegative real numbers by $\RR_\ge$.
For $x \in \RR^n$, we write $x \ge 0$ if $x \in \RR^n_\ge$.
Further, we denote the {\em support} of a vector $x \in \RR^n$ by 
$\supp(x) = \{ i \mid x_i \neq 0 \}$.

\subsection*{Sign vectors}

For $x \in \RR^n$, we define the {\em sign vector} $\sign(x) \in \{-,0,+\}^n$ 
by applying the sign function component-wise, that is,
$\sign(x)_i=\sign(x_i)$ for $i=1,\ldots,n$. 
The relations $0<-$ and $0<+$ induce a partial order on $\{-,0,+\}^n$:
for $X, Y \in \{-,0,+\}^n$,
we write $X \le Y$
if the inequality holds component-wise.
For $x,y \in \RR^n$,
we say that $x$ {\em conforms to} $y$, if $\sign(x) \le \sign(y)$.
For example,
let $x=(-1,0,2)^T$ and $y=(-2,-1,1)$. Then,
\[
\sign \!
\begin{pmatrix}
-1 \\ 0 \\ 2
\end{pmatrix}
=
\begin{pmatrix}
- \\ 0 \\ +
\end{pmatrix}
\le
\begin{pmatrix}
- \\ - \\ +
\end{pmatrix}
=
\sign \!
\begin{pmatrix}
-2 \\ -1 \\ 1
\end{pmatrix} ,
\]
that is, $\sign(x) \le \sign(y)$, and $x$ conforms to $y$.
Let $X \in \{-,0,+\}^n$. The corresponding closed orthant $O \subset \RR^n$ is defined as $O = \{ x \mid \sign(x) \le X \}$.

\subsection*{Convex cones}

A nonempty subset $C$ of a vector space is a {\em convex cone}, if
\[
x, y \in C \text{ and } \mu, \nu > 0 \text{ imply } \mu x + \nu y \in C,
\]
or, equivalently, if
\[
\lambda C = C \text{ for all } \lambda > 0 \text{ and } C+C=C.
\]
A convex cone $C$ is called {\em pointed} if $C \cap -C= \{0\}$.
It is {\em polyhedral} if
\[
C = \{ x \mid A x \ge 0 \} \quad \text{for some } A \in \RR^{m \times r} ,
\]
that is, if it is defined by finitely many homogeneous inequalities.
Hence, a polyhedral cone is pointed if and only if $\ker(A) = \{0\}$.

\subsection*{Special vectors}

We recall the definitions of support-minimal vectors and extreme vectors,
which play an important role in both polyhedral geometry and metabolic pathway analysis.
We also introduce support-wise non-decomposable vectors,
which serve as elementary modes for flux cones (in the original definition),
and conformally non-decomposable vectors,
which serve as elementary vectors for general polyhedral cones (see Subsection~\ref{sec:genpolcon}).

Let $C$ be a convex cone.
A nonzero vector $x \in C$ is called
\begin{itemize}
\item
\emph{support-minimal}, if
\begin{alignat}{1} \label{smv}
& \text{for all nonzero } x' \in C, \nonumber \\
& \supp(x') \subseteq \supp(x) \text{ implies } \supp(x') = \supp(x), \tag{SM}
\end{alignat}
\item
\emph{support-wise non-decomposable}, if
\begin{alignat}{1} \label{sndv}
& \text{for all nonzero } x^1, x^2 \in C \text{ with } \supp(x^1), \, \supp(x^2) \subseteq \supp(x), \nonumber \\
& x = x^1 + x^2 \text{ implies } \supp(x^1) = \supp(x^2), \tag{swND}
\end{alignat}
\item
\emph{conformally non-decomposable}, if
\begin{alignat}{1} \label{cndv}
& \text{for all nonzero } x^1, x^2 \in C \text{ with } \sign(x^1), \, \sign(x^2) \le \sign(x), \nonumber \\
& x = x^1 + x^2 \text{ implies } x^1 = \lambda x^2 \text{ with } \lambda > 0, \tag{cND}
\end{alignat}
\item
and \emph{extreme}, if
\begin{alignat}{1} \label{xv}
& \text{for all nonzero } x^1, x^2 \in C, \nonumber \\
& x = x^1 + x^2 \text{ implies } x^1 = \lambda x^2 \text{ with } \lambda > 0. \tag{EX}
\end{alignat}
\end{itemize}

From the definitions, we have the implications
\[
\text{SM} \Rightarrow \text{swND} \Leftarrow \text{EX} \Rightarrow \text{cND}.
\]

If $x \in C$ is extreme, then $\{ \lambda x \mid \lambda > 0 \}$ is called an extreme ray of $C$.
In fact, $C$ has an extreme ray if and only if $C$ is pointed.
If $C$ is contained in a closed orthant (and hence pointed),
we have the equivalence $\text{cND} \Leftrightarrow \text{EX}$.


\section{Mathematical results}

We start by extending a result on conformal decompositions into elementary vectors
from linear subspaces to special cases of polyhedral cones,
including flux cones in metabolic pathway analysis.

\subsection{Linear subspaces and s-cones}

\newcommand{\xym}{\left( \begin{smallmatrix} x \\ y \end{smallmatrix} \right)}

We consider linear subspaces with optional nonnegativity constraints as special cases of polyhedral cones.
Let $S \subseteq \RR^r$ be a linear subspace and $0 \le d \le r$.
We define the resulting s-cone (subspace cone, special cone) as
\[
C(S,d) = \{ \xym \in \RR^{(r-d)+d} \mid \xym \in S, \, y \ge 0 \}.
\]
Clearly, $C(S,0) = S$ and $C(S,r) = S \cap \RR^r_\ge$.

\begin{dfn}
Let $C(S,d)$ be an s-cone.
A vector $e \in C(S,d)$ is called {\em elementary} if it is support-minimal.
\end{dfn}

For linear subspaces, the definition of elementary vectors (EVs) as SM vectors was given in \cite{Rockafellar1969}.
For flux cones, where $S = \ker(N)$,
the definition of elementary modes (EMs) as SM vectors was given in \cite{SchusterHilgetagWoodsFell2002}.
Interestingly, the choice of the same adjective for the closely related concepts
of elementary vectors and elementary modes was coincidental~\cite{Schuster}.

In the proofs of Theorem~\ref{thm} and Propositions~\ref{pro:fin} and~\ref{pro:equ},
we use the following argument.

\begin{lem} \label{lem:arg}
Let $C(S,d)$ be an s-cone
and $x,x' \in C(S,d)$ be nonzero vectors which are not proportional.
If $\supp(x') \subseteq \supp(x)$,
then there exists a nonzero vector
\[
x'' = x - \lambda x' \in C(S,d) \quad \text{with } \lambda \in \RR
\]
such that
\[
\sign(x'') \le \sign(x) \quad \text{and} \quad \supp(x'') \subset \supp(x).
\]
If $\sign(x') \le \sign(x)$, then $\lambda > 0$ in $x''$.
\end{lem}
\begin{proof}
Clearly, $x''=x-\lambda x'$ is nonzero for all $\lambda \in \RR$.
There exists a largest~$\lambda > 0$ (in case $\sign(-x') \le \sign(x)$ a smallest~$\lambda < 0$)
such that $\sign(x'') \le \sign(x)$.
For this~$\lambda$, $x'' \in C(S,d)$ and $\supp(x'') \subset \supp(x)$.
\end{proof}

For linear subspaces, the following fundamental result was proved in \cite[Theorem~1]{Rockafellar1969}.
We extend it to s-cones.
\begin{thm} \label{thm}
Let $C(S,d)$ be an s-cone.
Every nonzero vector $x \in C(S,d)$ is a conformal sum of EVs.
That is,
there exists a finite set $E \subseteq C(S,d)$ of EVs such that
\[
x = \sum_{e \in E} e \quad \text{with } \sign(e) \le \sign(x).
\]
The set $E$ can be chosen such that its elements are linearly independent,
in particular,
they can be ordered such that every $e \in E$ has a component which is nonzero in $e$,
but zero in its predecessors (in the ordered set).
Then, $|E| \le \dim(S)$ and $|E| \le |\supp(x)|$.
\end{thm}
\begin{proof}
We proceed by induction on the cardinality of $\supp(x)$. \par
Either, $x$ is SM (and $E = \{x\}$)
or there exists a nonzero vector $x' \in C(S,d)$
with $\supp(x') \subset \supp(x)$, but not necessarily with $\sign(x') \le \sign(x)$.
However, by Lemma~\ref{lem:arg}, 
there exists a nonzero vector $x'' \in C(S,d)$ with $\sign(x'') \le \sign(x)$ and $\supp(x'') \subset \supp(x)$.
By the induction hypothesis,
there exists a SM vector $e^*$ with $\sign(e^*) \le \sign(x'')$ and hence $\sign(e^*) \le \sign(x)$.
By Lemma~\ref{lem:arg} again, 
there exists a nonzero vector
\[
x^* = x - \lambda e^* \in C(S,d) \quad \text{with } \lambda>0
\]
such that  $\sign(x^*) \le \sign(x)$ and $\supp(x^*) \subset \supp(x)$.
By the induction hypothesis,
there exists a finite set $E^*$ of SM vectors such that
\[
x^* = \sum_{e \in E^*} e \quad \text{with } \sign(e) \le \sign(x^*)
\]
and hence $\sign(e) \le \sign(x)$.
We have constructed a finite set $E = E^* \cup \{\lambda e^*\}$ of SM vectors such that
\[
x = x^* + \lambda e^* = \sum_{e \in E^*} e + \lambda e^* = \sum_{e \in E} e \quad \text{with } \sign(e) \le \sign(x) .
\]
By the induction hypothesis,
the set $E^*$ can be chosen such that its elements are linearly independent
and ordered such that every $e \in E^*$ has a component which is nonzero in $e$, but zero in all its predecessors.
By construction, $\lambda e^*$ has a component which is nonzero, but zero in $x^*$ and hence in all $e \in E^*$.
Obviously, the elements of $E=E^* \cup \{\lambda e^*\}$ are linearly independent and can be ordered accordingly.
\end{proof}

\renewcommand{\thefootnote}{\fnsymbol{footnote}}

The statement about the support of the EVs was too strong in \cite[Theorem~1]{Rockafellar1969}.
It was claimed that every EV has a nonzero component which is zero in all other EVs.\footnote{
For a counterexample, consider the subspace $S = \ker(1,-1,-1,1) \subseteq \RR^4$.
Its nonnegative EVs are
\[
e^1 = \begin{pmatrix} 1 \\ 1 \\ 0 \\ 0 \end{pmatrix}, \,
e^2 = \begin{pmatrix} 1 \\ 0 \\ 1 \\ 0 \end{pmatrix}, \,
e^3 = \begin{pmatrix} 0 \\ 1 \\ 0 \\ 1 \end{pmatrix}, \,
e^4 = \begin{pmatrix} 0 \\ 0 \\ 1 \\ 1 \end{pmatrix},
\]
and their positive multiples.
Then $x=(1,2,3,4)^T$ is not a conformal sum of EVs with the claimed property.
(Every conformal decomposition of $x$ consists of at least 3 EVs,
and every set of 3 EVs contains 1 EV which does not have a nonzero component which is zero in the other EVs.)
}

Theorem~\ref{thm} is a conformal refinement of Minkowski's and Carath\'eodory's theorems for s-cones.
In fact, it remains to show that there are finitely many EVs.

\begin{pro} \label{pro:fin}
Let $C(S,d)$ be an s-cone.
If two SM vectors $x,x' \in C(S,d)$ have the same sign vector, $\sign(x) = \sign(x')$, then $x = \lambda x'$ with $\lambda > 0$.
As a consequence, there are finitely many SM vectors up to positive scalar multiples.
\end{pro}
\begin{proof}
Assume there are two SM vectors with the same sign vector which are not proportional.
Then, by Lemma~\ref{lem:arg}, there exists a vector with smaller support.
\end{proof}

We conclude by showing that, for s-cones,
EVs can be equivalently defined as SM, swND, or cND vectors.

\begin{pro} \label{pro:equ}
For an s-cone, support-minimality, support-wise non-decompos\-ability, and conformal non-decomposability are equivalent.
That is,
\[
\text{s-cone} \colon \quad \text{SM} \Leftrightarrow \text{swND} \Leftrightarrow \text{cND}.
\]
\end{pro}
\begin{proof}
$\text{SM} \Rightarrow \text{swND}$:
By definition. \par
$\text{swND} \Rightarrow \text{cND}$:
Let $C(S,d)$ be an s-cone
and assume that $x \in C(S,d)$ is conformally decomposable, that is,
$x=x^1+x^2$ with nonzero $x^1, x^2 \in C(S,d)$, $\sign(x^1), \sign(x^2) \le \sign(x)$, and $x^1, x^2$ being not proportional.
By Lemma~\ref{lem:arg}, 
there exists a nonzero $x'=x - \lambda x^1 \in C(S,d)$ 
such that $\supp(x') \subset \supp(x)$. Hence $\supp(x') \neq \supp(x^1)$,
and $x = x' + \lambda x^1$ is support-wise decomposable.

$\text{cND} \Rightarrow \text{SM}$:
Let $C(S,d)$ be an s-cone
and assume that $x \in C(S,d)$ is not SM,
that is, there exists a nonzero $x' \in C(S,d)$ with $\supp(x') \subset \supp(x)$.
Then, there exists a largest $\lambda > 0$ such that
$x^1 = \frac{1}{2} x + \lambda x'$ and $x^2 = \frac{1}{2} x - \lambda x'$ fulfill $\sign(x^1), \, \sign(x^2) \le \sign(x)$.
For this $\lambda$,
either $\supp(x^1) \subset \supp(x)$ or $\supp(x^2) \subset \supp(x)$;
in any case, $x^1, x^2 \in C(S,d)$ and $\supp(x^1) \neq \supp (x^2)$.
Hence, $x = x^1+x^2$ is conformally decomposable.
\end{proof}

If an s-cone is contained in a closed orthant, then further $\text{cND} \Leftrightarrow \text{EX}$,
and all definitions of special vectors are equivalent.


\newcommand{\xAx}[1]{\left( \begin{smallmatrix} x^{#1} \\ Ax^{#1} \end{smallmatrix} \right)}
\newcommand{\eAe}[1]{\left( \begin{smallmatrix} e^{#1} \\ Ae^{#1} \end{smallmatrix} \right)}


\subsection{General polyhedral cones} \label{sec:genpolcon}

Let $C$ be a polyhedral cone, that is,
\[
C = \{ x \in \RR^r \mid A x \ge 0 \} \quad \text{for some } A \in \RR^{m \times r}.
\]

For s-cones,
we defined elementary vectors (EVs) via support-minimality
which, in this case, turned out to be equivalent to conformal non-decomposability.
For general polyhedral cones,
only the latter concept allows to extend Theorem~\ref{thm}.

\begin{dfn}
Let $C$ be a polyhedral cone.
A vector $e \in C$ is called {\em elementary} if it is conformally non-decomposable.
\end{dfn}

In order to apply Theorem~\ref{thm},
we define an s-cone related to a polyhedral cone $C$.
We introduce the subspace
\[
\tilde S = \{ \xAx{} \in \RR^{r+m} \mid x \in \spann(C) \}
\]
with $\dim(\tilde S) = \dim(C)$
and the s-cone
\begin{align*}
\tilde C &= C(\tilde S, m) \\
&= \{ \xAx{} \in \RR^{r+m} \mid x \in \spann(C) \text{ and } Ax \ge 0 \} \\
&= \{ \xAx{} \in \RR^{r+m} \mid x \in C \} .
\end{align*}
Hence,
\[
x \in C
\quad \Leftrightarrow \quad
\xAx{} \in \tilde C.
\]
Moreover,
the cND vectors of $C$ and $\tilde C$ are in one-to-one correspondence.

\begin{lem} \label{lem:121}
Let $C = \{ x \mid A x \ge 0 \}$ be a polyhedral cone and $\tilde C = \{ \xAx{} \mid Ax \ge 0 \}$ the related s-cone.
Then,
\[
x \in C \text{ is cND}
\quad \Leftrightarrow \quad
\xAx{} \in \tilde C \text{ is cND}.
\]
\end{lem}
\begin{proof}
First, we show the equivalence of the premises in the definitions of conformal non-decomposability for $C$ and $\tilde C$.
Indeed,
\begin{gather*}
x = x^1+x^2
\text{ with } x^1, x^2 \in C \\ 
\Leftrightarrow \\
\xAx{} = \xAx{1}+\xAx{2}
\text{ with } \xAx{1} \! , \xAx{2} \in \tilde C.
\end{gather*}
Assuming $x = x^1+x^2$ with $x^1, x^2 \in C$ (and hence $Ax^1, Ax^2, Ax \ge 0$), we have
\[
\sign(x^1), \, \sign(x^2) \le \sign(x)
\quad \Leftrightarrow \quad
\sign \! \xAx{1} \! , \, \sign \! \xAx{2} \le \sign \! \xAx{}.
\]
It remains to show the equivalence of the conclusions in the two definitions.
In fact,
\[
x^1= \lambda x^2 \text{ with } \lambda > 0
\quad \Leftrightarrow \quad
\xAx{1} = \lambda \! \xAx{2} \text{ with } \lambda > 0.
\]
\end{proof}

\noindent
Now, we can extend Theorem~\ref{thm} to general polyhedral cones.

\begin{thm} \label{thm:polcon}
Let $C = \{ x \mid A x \ge 0 \}$ be a polyhedral cone.
Every nonzero vector $x \in C$ is a conformal sum of EVs.
That is,
there exists a finite set $E \subseteq C$ of EVs such that
\[
x = \sum_{e \in E} e \quad \text{with } \sign(e) \le \sign(x).
\]
The set $E$ can be chosen such that $|E| \le \dim(C)$ and $|E| \le |\supp(x)| + |\supp(Ax)|$.
\end{thm}
\begin{proof}
Let $A \in \RR^{m \times r}$.
Define the subspace
\[
\tilde S = \{ \xAx{} \in \RR^{r+m} \mid x \in \spann(C) \}
\]
and the s-cone
\[
\tilde C = \{ \xAx{} \in \RR^{r + m} \mid x \in C \}.
\]
Let $x \in C$ be nonzero.
By Theorem~\ref{thm}, $\xAx{} \in \tilde C$ is a conformal sum of EVs.
That is, there exists a finite set $\tilde E \subseteq \tilde C$ of EVs such that
\[
\xAx{} = \sum_{\eAe{} \in \tilde E} \eAe{} \quad \text{with } \sign \! \eAe{} \le \sign \! \xAx{}.
\]
By Lemma~\ref{lem:121},
the EVs of $C$ and $\tilde C$ are in one-to-one correspondence.
Hence, there exists a finite set $E = \{ e \mid \eAe{} \in \tilde E \} \subseteq C$ of EVs such that
\[
x = \sum_{e \in E} e \quad \text{with } \sign(e) \le \sign(x).
\]
The set $\tilde E$ (and hence $E$) can be chosen such that $|E| = |\tilde E| \le \dim(\tilde S) = \dim(C)$
and $|E| = |\tilde E| \le |\supp \! \xAx{}| = |\supp(x)|+|\supp(Ax)|$.
\end{proof}

Theorem~\ref{thm:polcon} is a conformal refinement of Minkowski's and Carath\'eodory's theorems for polyhedral cones.
In fact, it remains to show that there are finitely many EVs.

\begin{pro}
For a polyhedral cone,
there are finitely many cND vectors up to positive scalar multiples.
\end{pro}
\begin{proof}
Let $C$ be a polyhedral cone and $\tilde C$ the related s-cone.
By Lemma~\ref{lem:121}, the cND vectors of $C$ and $\tilde C$ are in one-to-one correspondence.
By Proposition~\ref{pro:equ}, the cND and SM vectors of $\tilde C$ coincide,
and by Proposition~\ref{pro:fin}, there are finitely many SM vectors.
\end{proof}

In~\cite{UrbanczikWagner2005},
EVs of a polyhedral cone $C$ were equivalently defined as extreme vectors of intersections of $C$
with closed orthants of maximal dimension.
Indeed, the following equivalence holds for closed orthants, not necessarily of maximal dimension.
\begin{pro}
Let $C \subseteq \RR^r$ be a polyhedral cone,
$x \in C$, and $O \subset \RR^r$ a closed orthant with $x \in O$.
Then,
\[
x \in C \text{ is cND}
\quad \Leftrightarrow \quad
x \in C \cap O \text{ is EX}.
\]
\end{pro}
\begin{proof}
We show the equivalence of the premises in the definitions of conformal non-decomposability for $C$
and extremity for $C \cap O$.
(The conclusions are identical.)
Indeed, assuming $x=x^1+x^2$, we have
\[
x^1,x^2 \in C \text{ with } \sign(x^1), \, \sign(x^2) \le \sign(x)
\quad \Leftrightarrow \quad
x^1, x^2 \in C \cap O.
\]
\end{proof}


\newcommand{\xxi}{\left( \begin{smallmatrix} x \\ \xi \end{smallmatrix} \right)}

\newcommand{\xAxx}[1]{\left( \begin{smallmatrix} x^{#1} \\ \xi^{#1} \\ Ax^{#1}-\xi^{#1} b \end{smallmatrix} \right)}
\newcommand{\xAxz}[1]{\left( \begin{smallmatrix} x^{#1} \\ 0 \\ Ax^{#1} \end{smallmatrix} \right)}
\newcommand{\xAxo}[1]{\left( \begin{smallmatrix} x^{#1} \\ 1 \\ Ax^{#1}-b \end{smallmatrix} \right)}

\newcommand{\eAez}[1]{\left( \begin{smallmatrix} e_{#1} \\ 0 \\ Ae_{#1} \end{smallmatrix} \right)}
\newcommand{\eAeo}[1]{\left( \begin{smallmatrix} e_{#1} \\ 1 \\ Ae_{#1}-b \end{smallmatrix} \right)}


\subsection{Polyhedra} \label{sec:pol}

Let $P$ be a polyhedron, that is,
\[
P = \{ x \in \RR^r \mid A x \ge b \} \quad \text{for some } A \in \RR^{m \times r} \text{ and } b \in \RR^m.
\]
In order to extend Theorem~\ref{thm} to polyhedra,
we introduce corresponding special vectors.

\subsubsection*{Special vectors}

Let $P$ be a polyhedron.
A vector $x \in P$ is called
\begin{itemize}
\item
a \emph{vertex}, if
\begin{alignat}{1} \label{ve}
& \text{for all } x^1, x^2 \in P \text{ and } 0<\lambda<1, \nonumber \\
& x = \lambda x^1 + (1-\lambda) x^2 \text{ implies } x^1=x^2, \tag{VE}
\end{alignat}
\item
and \emph{convex-conformally non-decomposable}, if
\begin{alignat}{1} \label{ccnd}
& \text{for all } x^1, x^2 \in P \text{ with } \sign(x^1), \, \sign(x^2) \le \sign(x) \text{ and } 0<\lambda<1, \nonumber \\
& x = \lambda x^1 + (1-\lambda) x^2 \text{ implies } x^1=x^2. \tag{ccND}
\end{alignat}
\end{itemize}

From the definitions, we have
\[
\text{VE} \Rightarrow \text{ccND}.
\]

\noindent
For a polyhedral cone, we defined elementary vectors (EVs) via conformal non-decomposability.
For a polyhedron, we require two sorts of EVs: convex-conformally non-decomposable vectors of the polyhedron
and conformally non-decomposable vectors of its recession cone.

\begin{dfn}
Let $P = \{ x \in \RR^r \mid Ax \ge b \}$ be a polyhedron and $C^r = \{ x \in \RR^r \mid Ax \ge 0 \}$ its recession cone.
A vector $e \in C^r \cup P$ is called an {\em elementary vector of $P$}
if either $e \in C^r$ is conformally non-decomposable
or $e \in P$ is convex-conformally non-decomposable.
\end{dfn}

In order to apply Theorem~\ref{thm},
we define an s-cone related to a polyhedron $P = \{ x \in \RR^r \mid A x \ge b \}$.
We introduce the {\em homogenization}
\[
C^h = \{ \xxi \in \RR^{r+1} \mid \xi \ge 0 \text{ and } Ax - \xi b \ge 0 \}
\]
of the polyhedron,
the subspace
\[
\tilde S = \{ \xAxx{} \in \RR^{r+1+m} \mid \xxi \in \spann(C^h) \}
\]
with $\dim(\tilde S) = \dim(C^h) = \dim(P)+1$,
and the s-cone
\begin{align*}
\tilde C &= C(\tilde S, 1+m) \\
&= \{ \xAxx{} \in \RR^{r+1+m} \mid \xxi \in \spann(C^h), \, \xi \ge 0, \text{ and } A x - \xi b \ge 0 \} \\
&= \{ \xAxx{} \in \RR^{r+1+m} \mid \xxi \in C^h \} .
\end{align*}
Hence,
\[
\xxi \in C^h
\quad \Leftrightarrow \quad
\xAxx{} \in \tilde C.
\]
Moreover, the cND vectors of $C^r$ and the ccND vectors of $P$
are in one-to-one correspondence with the cND vectors of $\tilde C$.

\begin{lem} \label{lem:121121}
Let $P = \{ x \mid A x \ge b \}$ be a polyhedron, $C^r = \{ x \mid Ax \ge 0 \}$ its recession cone,
and
\[
\tilde C = \{ \xAxx{} \in \RR^{r+1+m} \mid \xi \ge 0 \text{ and } Ax - \xi b \ge 0 \}
\]
the related s-cone.
Then,
\begin{gather*}
x \in C^r \text{ is cND}
\quad \Leftrightarrow \quad
\xAxz{} \in \tilde C \text{ is cND}
\intertext{and}
x \in P \text{ is ccND }
\quad \Leftrightarrow \quad
\xAxo{} \in \tilde C \text{ is cND}.
\end{gather*}
\end{lem}
\begin{proof}
See Appendix.
\end{proof}

\noindent
Now, we can extend Theorem~\ref{thm} to polyhedra.

\begin{thm} \label{thm:pol}
Let $P = \{ x \mid A x \ge b \}$ be a polyhedron and $C^r = \{ x \mid Ax \ge 0 \}$ its recession cone.
Every vector $x \in P$ is a conformal sum of EVs.
That is,
there exist finite sets $E_0 \subseteq C^r$ and $E_1 \subseteq P$ of EVs such that
\[
x = \sum_{e \in E_0} e + \sum_{e \in E_1} \lambda_e e \quad \text{with } \sign(e) \le \sign(x),
\]
$\lambda_e \ge 0$, and $\sum_{e \in E_1} \lambda_e = 1$.
(Hence, $|E_1| \ge 1$.)
\par
The set $E = E_0 \cup E_1$ can be chosen such that $|E| \le \dim(P)+1$ and $|E| \le |\supp(x)| + |\supp(Ax)| + 1$.
\end{thm}
\begin{proof}
By defining an s-cone related to $P$, applying Theorem~\ref{thm}, and using Lemma~\ref{lem:121121}.
See Appendix.
\end{proof}

Theorem~\ref{thm:pol} is a conformal refinement of Minkowski's and Carath\'eodory's theorems for polyhedra.
In fact, it remains to show that there are finitely many EVs.

\begin{pro}
For a polyhedron,
there are finitely many ccND vectors.
\end{pro}
\begin{proof}
Let $P$ be a polyhedron and $\tilde C$ the related s-cone.
By Lemma~\ref{lem:121121}, the ccND vectors of $P$ are in one-to-one correspondence
with a subset of cND vectors of $\tilde C$.
By Proposition~\ref{pro:equ}, the cND and SM vectors of $\tilde C$ coincide,
and by Proposition~\ref{pro:fin}, there are finitely many SM vectors.
\end{proof}

EVs of a polyhedron $P$ can be equivalently defined as vertices of intersections of $P$
with closed orthants.
\begin{pro}
Let $P \subseteq \RR^r$ be a polyhedron,
$x \in P$, and $O \subset \RR^r$ a closed orthant with $x \in O$.
Then,
\[
x \in P \text{ is ccND}
\quad \Leftrightarrow \quad
x \in P \cap O \text{ is VE}.
\]
\end{pro}
\begin{proof}
We show the equivalence of the premises in the definitions
of convex-conformal non-decomposability for $P$
and of a vertex for $P \cap O$.
(The conclusions are identical.)
Indeed, assuming $x=\lambda x^1+(1-\lambda) x^2$ with $0<\lambda<1$, we have
\[
x^1,x^2 \in P \text{ with } \sign(x^1), \, \sign(x^2) \le \sign(x)
\quad \Leftrightarrow \quad
x^1, x^2 \in P \cap O.
\]
\end{proof}

We conclude by noting that Theorem~\ref{thm:polcon} is a special case of Theorem~\ref{thm:pol}.
If a polyhedron is also a cone, then $P=C^r$, $E_1=\{0\}$, and $\sum_{e \in E_1} \lambda_e e = 0$.
However, we do not use Theorem~\ref{thm:polcon} to prove Theorem~\ref{thm:pol}.
In classical proofs of Minkowski's and Carath\'eodory's theorems,
one first studies polyhedral cones and then extends the results to polyhedra by a method called homogenization/dehomogenization; see e.g.~\cite{Ziegler1995}.

\subsection{Minimal generating sets}

For a pointed polyhedral cone,
the extreme rays form a minimal set of generators with respect to addition.
The set is minimal in the sense that no proper subset forms a generating set
and minimal in the even stronger sense that it is contained in every other generating set.
Hence, the extreme rays form a {\em unique} minimal set of generators.

For a general polyhedral cone,
there are minimal sets of generators (minimal in the sense that no proper subset forms a generating set),
but there is no unique minimal generating set.
However, there is a unique minimal set of {\em conformal} generators,
namely the set of elementary vectors.

Recall that elementary vectors of a polyhedral cone are defined as conformally non-decomposable vectors.
Indeed, every nonzero element of a polyhedral cone is a conformal sum of elementary vectors (Theorem \ref{thm:polcon}),
and every elementary vector is contained in a set of conformal generators.

We make the above argument more formal.

\begin{dfn}
Let $C$ be a polyhedral cone. A subset $G \subseteq C$ is called a \emph{conformal generating set}
if (i) every nonzero vector $x \in C$ is a conformal sum of vectors in $G$,
that is, if there exists a finite set $G_x \subset G$ such that
\[
x = \sum_{g \in G_x} g \quad \text{with } \sign(g) \le \sign(x) ,
\]
and (ii) if $\lambda G=G$ for all $\lambda>0$.
\end{dfn}

\begin{pro}
\label{pro:min}
Let $C$ be a polyhedral cone,
$E \subseteq C$ the set of elementary vectors, and $G \subseteq C$ a conformal generating set.
Then, $E \subseteq G$.
\end{pro}
\begin{proof}
Let $e\in C$ be an elementary vector. Since $G$ is a conformal generating set, we have
\[
e = g^*+h \quad \text{with } \sign(g^*), \, \sign(h) \le \sign(x),
\]
where we choose a nonzero $g^* \in G_e \subset G$ and set $h = \sum_{g \in G_e \setminus \{g^*\}} g \in C$.
If $|G_e|=1$, then $h=0$ and $e=g^*\in G$.
Otherwise, since $e$ is an elementary vector (a cND vector),
we have $h = \lambda g^*$ with $\lambda>0$ and hence $e = (1+\lambda) g^* \in G$.
\end{proof}

Analogously, for a polyhedron,
there is a unique minimal set of conformal  generators,
namely the set of elementary vectors.

\subsection{Examples}

We illustrate our results by examples of polyhedral cones and polyhedra in two dimensions,
and we return to the running example from the introduction.

\begin{exa}
The s-cone $C = \{ x \mid x_1\ge0 , \, x_2\ge0 \}$.
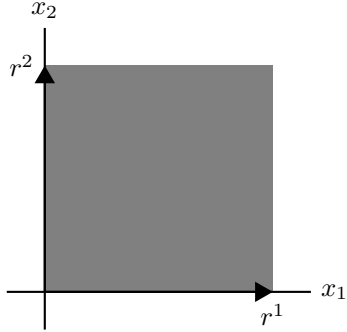
\begin{center}
\begin{tikzpicture}[scale=2]
\fill[fill=gray] (0,0) -- (1.5,0) -- (1.5,1.5) -- (0,1.5) -- cycle;
\draw[thick] (-0.25,0) -- (1.75,0) node[right] {$x_1$};
\draw[thick] (0,-0.25) -- (0,1.75) node[above] {$x_2$};
\draw[thick,->,>=triangle 60] (0,0) -- (1.5,0) node[yshift=-2ex] {$r^1$};
\draw[thick,->,>=triangle 60] (0,0) -- (0,1.5) node[xshift=-2ex] {$r^2$};
\end{tikzpicture}
\end{center}
Its EVs (SM vectors) are elements of the rays $r^1 = \{ x \mid x_1>0 , \, x_2=0 \}$ and $r^2 = \{ x \mid x_1=0 , \, x_2>0 \}$
(indicated by arrows).
Every nonzero vector $x \in C$ is a conformal sum of EVs.
That is,
\[
x = e^1 + e^2 ,
\]
where $e^1 \in r^1$ and $e^2 \in r^2$.
\end{exa}

\newpage

\begin{exa}
The general polyhedral cone $C = \{ x \mid \begin{pmatrix} 3 & 1 \\ -1 & 1 \end{pmatrix}
\begin{pmatrix} x_1 \\ x_2 \end{pmatrix} \ge 0 \}$.
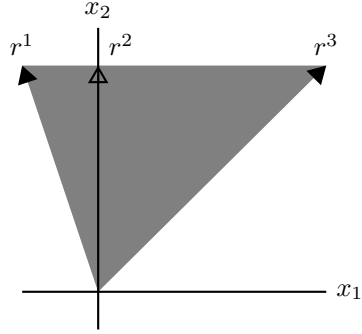
\begin{center}
\begin{tikzpicture}[scale=2]
\fill[fill=gray] (-0.5,1.5) -- (0,0) -- (1.5,1.5) -- cycle;
\draw[thick] (-0.5,0) -- (1.5,0) node[right] {$x_1$};
\draw[thick] (0,-0.25) -- (0,1.75) node[above] {$x_2$};
\draw[thick,->,>=triangle 60] (-0.475,1.425) -- (-0.5,1.5) node[above] {$r^1$};
\draw[thick,->,>=open triangle 60,fill=white] (0,1.45) -- (0,1.5) node[xshift=2ex,above] {$r^2$};
\draw[thick,->,>=triangle 60] (1.45,1.45) -- (1.5,1.5) node[above] {$r^3$};
\end{tikzpicture}
\end{center}
Its EVs (cND vectors) are elements of the rays $r^1$, $r^2$, and $r^3$.
Note that $r^2$ is not an extreme ray of $C$,
but an extreme ray of $C \cap \RR^2_\ge$,
the intersection of the cone with the nonnegative orthant.
Every nonzero vector $x \in C$ is a conformal sum of EVs.
In particular, if $x \in C \cap \RR^2_\ge$, then
\[
x = e^2 + e^3 ,
\]
where $e^2 \in r^2$ and $e^3 \in r^3$.
\end{exa}


\begin{exa}
The polyhedron $P= \{ x \mid \begin{pmatrix} 3 & 1 \\ -3 & 3 \\ 0 & 2 \end{pmatrix}
\begin{pmatrix} x_1 \\ x_2 \end{pmatrix} \ge \begin{pmatrix} 1 \\ -1 \\ 1 \end{pmatrix} \}$.
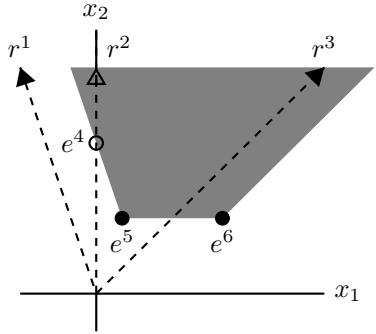
\begin{center}
\begin{tikzpicture}[scale=2]
\fill[fill=gray] (-0.17,1.5) -- (0.17,0.5) -- (0.83,0.5) -- (1.83,1.5) -- cycle;
\draw[thick,dashed] (-0.5,1.5) -- (0,0) -- (1.5,1.5);
\draw[thick] (-0.5,0) -- (1.5,0) node[right] {$x_1$};
\draw[thick] (0,-0.25) -- (0,0);
\draw[thick] (0,1.5) -- (0,1.75);
\draw[thick,dashed] (0,0) -- (0,1.75) node[above] {$x_2$};
\draw[thick,->,>=triangle 60] (-0.475,1.425) -- (-0.5,1.5) node[above] {$r^1$};
\draw[thick,->,>=open triangle 60] (0,1.45) -- (0,1.5) node[xshift=2ex,above] {$r^2$};
\draw[thick,->,>=triangle 60] (1.45,1.45) -- (1.5,1.5) node[above] {$r^3$};
\draw[thick,-o] (0,0.95) -- (0,1.05) node[xshift=-2ex,yshift=-0.5ex] {$e^4$};
\draw[fill=black] (0.17,0.5) circle (0.3ex) node[yshift=-2ex] {$e^5$};
\draw[fill=black] (0.83,0.5) circle (0.3ex) node[yshift=-2ex] {$e^6$};
\end{tikzpicture}
\end{center}
Its EVs are elements of the rays $r^1$, $r^2$, and $r^3$ (cND vectors of the recession cone)
and the vectors $e^4$, $e^5$, and $e^6$ (ccND vectors of the polyhedron).
Note that $e^4$ is not a vertex of $P$,
but a vertex of $P \cap \RR^2_\ge$,
the intersection of the polyhedron with the nonnegative orthant.
Every vector $x \in P$ is a conformal sum of EVs.
In particular, if $x \in P \cap \RR^2_\ge$, then
\[
x = (e^2 + e^3) + (\lambda_4 e^4 + \lambda_5 e^5 + \lambda_6 e^6) ,
\]
where $e^2 \in r^2$, $e^3 \in r^3$ and $\lambda_4,\lambda_5,\lambda_6 \ge 0$ with $\lambda_4 + \lambda_5 + \lambda_6 = 1$.
\end{exa}

\newpage

Finally, we return to the running example from the introduction.
We restate the underlying network, the corresponding stoichiometric matrix and the resulting flux cone:

\bigskip
\hspace{-0.05\textwidth}
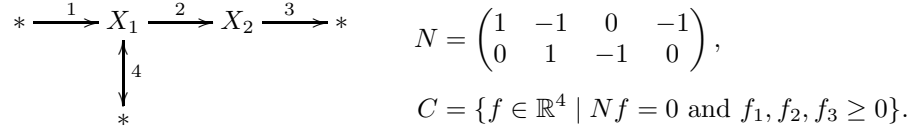
\begin{minipage}[c]{0.40\textwidth}
\xymatrix{
\ast \ar[r]^1 & X_1 \ar[r]^2 \ar@{<->}[d]^4 & X_2 \ar[r]^3 & \ast \\
& \ast
}
\end{minipage}
\hfill
\begin{minipage}[c]{0.55\textwidth}
\vspace{-2.ex}
\begin{align*}
N &=
\begin{pmatrix}
1 & -1 & 0 & -1 \\
0 & 1 & -1 & 0
\end{pmatrix}  , \\[1ex]
C &= \{ f \in \RR^4 \mid N f = 0 \text{ and } f_1,f_2,f_3 \ge 0 \} .
\end{align*}
\end{minipage}
\bigskip

\noindent
Its EVs (SM vectors) are 
\[
e^1 = \begin{pmatrix} 1 \\ 0 \\ 0 \\ 1 \end{pmatrix}, \,
e^2 = \begin{pmatrix} 0 \\ 1 \\ 1 \\ -1 \end{pmatrix}, \,
e^3 = \begin{pmatrix} 1 \\ 1 \\ 1 \\ 0 \end{pmatrix},
\]
and their positive multiples.
In other words, the EVs are elements of the rays $r^1 = \{ \lambda \, e^1 \mid \lambda>0 \}$,
$r^2 = \{ \lambda \, e^2 \mid \lambda>0 \}$, and $r^3 = \{ \lambda \, e^3 \mid \lambda>0 \}$.

The flux cone is defined by the stoichiometric matrix and the set of irreversible reactions.
If additionally lower/upper bounds for the fluxes through certain reactions are known,
then one is interested in the resulting flux polyhedron.
In the example, we add an upper bound for the flux through reaction~1,
in particular, we require $f_1 \le 2$ and obtain the flux polyhedron
\[
P = \{ f \in \RR^4 \mid N f = 0, \, f_1,f_2,f_3 \ge 0 , \text{ and } f_1 \le 2 \} .
\]
Its EVs are elements of the ray $r^2 = \{ \lambda \, e^2 \mid \lambda>0 \}$ (cND vectors of the recession cone)
and the vectors $e^1, e^3, e^4$ (ccND vectors of the polyhedron), where
\[
e^1 = \begin{pmatrix} 2 \\ 0 \\ 0 \\ 2 \end{pmatrix}, \,
e^2 = \begin{pmatrix} 0 \\ 1 \\ 1 \\ -1 \end{pmatrix}, \,
e^3 = \begin{pmatrix} 2 \\ 2 \\ 2 \\ 0 \end{pmatrix}, \,
e^4 = \begin{pmatrix} 0 \\ 0 \\ 0 \\ 0 \end{pmatrix}.
\]
Note that $e^3$ is not a vertex of $P$,
but a vertex of $P \cap \RR^4_\ge$,
the intersection of the polyhedron with the nonnegative orthant.
Every vector $x \in P$ is a conformal sum of EVs.
In particular, if $x \in P \cap \RR^4_\ge$, then
\[
x = \lambda_1 e^1 + \lambda_3 e^3 + \lambda_4 e^4 ,
\]
where $\lambda_1,\lambda_3,\lambda_4 \ge 0$ with $\lambda_1 + \lambda_3 + \lambda_4 = 1$.
In other words, the polyhedron $P \cap \RR^4_\ge$ is a polytope.

In applications such as computational strain design,
the set of EVs (the unique minimal set of {\em conformal} generators) is often more useful than a minimal set of generators.
In the example, the set of EVs includes $e^3$ which is a ccND vector, but not a vertex of $P$.
If we delete reaction~4 by gene knockout,
the new set of EVs consists of $e^3$ and $e^4$ (having zero flux through reaction 4),
and the resulting flux polyhedron is the polytope generated by $e^3$ and $e^4$.
Most importantly, we obtain the result without recalculating the set of generators
(after deleting reaction~4).


\section{Discussion}

Metabolic pathway analysis aims to identify meaningful routes in a network,
in particular, to decompose fluxes into {\em minimal} metabolic pathways.
However, only a decomposition {\em without cancelations} is biochemically meaningful,
since a reversible reaction cannot have a flux in different directions at the same time.

In mathematical terms, one is interested in a {\em conformal} decomposition of the flux cone
and of general polyhedral cones and polyhedra.
In this work, we first study s-cones (like the flux cone)
arising from a linear subspace and nonnegativity conditions.
Then, we analyze general polyhedral cones and polyhedra
via corresponding higher-dimensional s-cones.
Without assuming previous know\-ledge of polyhedral geometry,
we provide an elementary proof of a conformal refinement of Minkowski's and Carath\'eodory's theorems
(Theorems~\ref{thm}, \ref{thm:polcon}, and~\ref{thm:pol}):
Every vector (of an s-cone, a general polyhedral cone, or a polyhedron)
is a conformal sum of {\em elementary vectors} (conformally non-decomposable vectors),
and there is an upper bound on the number of elementary vectors needed in a conformal decomposition
(in terms of the dimension of the cone or polyhedron).

As a natural next question, one may ask: what is a {\em minimal generating set} of a polyhedral cone
that allows a conformal decomposition of every vector?
Clearly, such a set must contain all conformally non-decomposable vectors.
Indeed, we show that the elementary vectors form a {\em unique} minimal set of conformal generators (Proposition~\ref{pro:min}).
In metabolic pathway analysis, the question is: what is a minimal generating set of the flux cone
that allows a biochemically meaningful decomposition of every flux mode?
In this case, the {\em elementary modes} form a unique minimal set of generators without cancelations. 
This property distinguishes elementary modes as a fundamental concept in metabolic pathway analysis
and may serve as a definition.

The correspondence of general polyhedral cones and polyhedra to higher-dimensional s-cones
has also important consequences for the {\em computation} of elementary vectors.
In particular, it allows to use efficient algorithms and software developed for elementary modes
(see e.g.~\cite{Zanghellini2013} and the references therein)
for computing elementary vectors of general polyhedral cones and polyhedra.

In {\em applications},
decompositions without cancelations were first used in the study of the conversion cone~\cite{UrbanczikWagner2005},
a general polyhedral cone obtained by flux cone projection~\cite{Marashi2012}.
The approach was extended to polyhedra
arising from the flux cone and inhomogeneous constraints,
in particular, to describe the solution set of linear optimization problems encountered in flux balance analysis~\cite{Urbanczik2007}.
In analogy to s-cones, these sets could be called s-polyhedra.
Recently, elementary vectors have been used to describe such polyhedra
in the study of growth-coupled product synthesis~\cite{KlamtMahadevan2015}.
Interestingly,
conformal decompositions of the flux cone itself appeared rather late.
In fact, they have been used to characterize optimal solutions of enzyme allocation problems
in {\em kinetic} metabolic networks~\cite{MuellerRegensburgerSteuer2014}.

Minkowski's and Carath\'eodory's theorems (and their conformal refinements) are fundamental results in polyhedral geometry
with important applications in metabolic pathway analysis.
In subsequent work, we plan to revisit other results from polyhedral geometry and oriented matroids (like Farkas' lemma)
and investigate their consequences for metabolic pathway analysis.


\subsubsection*{Ackowledgments}

SM was supported by the Austrian Science Fund (FWF), project P28406.
GR was supported by the FWF, project P27229.


\section*{Appendix}

We prove the main results for polyhedra, Lemma~\ref{lem:121121} and Theorem~\ref{thm:pol}.

\begin{proof}[Proof of Lemma~\ref{lem:121121}]
To prove the first equivalence,
we note that $\xAxz{} \in \tilde C$ is cND if and only if $\xAx{} \in C'$ is cND, where $C' = \{ \xAx{} \in \RR^{r+m} \mid Ax \ge 0 \}$,
and apply Lemma~\ref{lem:121}.

To prove the second equivalence,
we show the two implications separately: \\
($\Rightarrow$)
We assume that $x \in P$ is ccND
and first consider a conformal sum of the form
\[
\xAxo{} = \xAxo{1} + \xAxz{2}
\]
with $x^1 \in P$, nonzero $x^2 \in C^r$, and $\sign(x^1), \, \sign(x^2) \le \sign(x)$.
As a matter of fact,
we also have $x = \frac{1}{2} x^1 + \frac{1}{2} (x^1+2x^2)$ with $x^1, x^1+2x^2 \in P$ and $\sign(x^1), \, \sign(x^1+2x^2) \le \sign(x)$.
By the assumption, $x^1 = x^1+2x^2$, that is, $x^2=0$, and it remains to consider a conformal sum of the form
\[
\xAxo{} = \lambda \! \xAxo{1} + (1 - \lambda) \! \xAxo{2} \tag{+}
\]
with $x^1,x^2 \in P$, $\sign(x^1), \, \sign(x^2) \le \sign(x)$, and $0<\lambda<1$.
By the assumption, $x^1=x^2$,
and the first vector in the sum is a positive multiple of the second.
That is,
\[
\lambda \! \xAxo{1} = \mu \, (1 - \lambda) \! \xAxo{2} \tag{$\ast$}
\]
with $\mu>0$.
Hence, $\xAxo{} \in \tilde C$ is cND. \\
($\Leftarrow$)
We assume that $\xAxo{} \in \tilde C$ is cND and consider the convex-conformal sum
\[
x = \lambda x^1 + (1-\lambda) x^2
\]
with $x^1,x^2 \in P$, $\sign(x^1), \, \sign(x^2) \le \sign(x)$, and $0<\lambda<1$.
Hence, we also have the conformal sum (+).
By the assumption,
we have equation ($\ast$) which implies $x^1=x^2$.
Hence, $x \in P$ is ccND.
\end{proof}

\begin{proof}[Proof of Theorem~\ref{thm:pol}]
Let $A \in \RR^{m \times r}$ and $b \in \RR^m$.
Define the homogenization
\[
C^h = \{ \xxi \in \RR^{r+1} \mid \xi \ge 0 \text{ and } Ax - \xi b \ge 0 \} ,
\]
the subspace
\[
\tilde S = \{ \xAxx{} \in \RR^{r+1+m} \mid \xxi \in \spann(C^h) \}
\]
and the s-cone
\[
\tilde C = \{ \xAxx{} \in \RR^{r+1+m} \mid \xxi \in C^h \} .
\]
Let $x \in P$.
By Theorem~\ref{thm}, $\xAxo{} \in \tilde C$ is a conformal sum of EVs.
That is, there exist finite sets $\tilde E_0, \tilde E_1 \subseteq \tilde C$ of (normalized) EVs such that
\[
\xAxo{} = \sum_{\eAez{} \in \tilde E_0} \eAez{} + \sum_{\eAeo{} \in \tilde E_1} \lambda_e \eAeo{}
\]
with
\[
\sign \! \eAez{} \!, \, \sign \! \eAeo{} \le \sign \! \xAxo{} ,
\]
$\lambda_e \ge 0$, and $\sum_{e \in E_1} \lambda_e = 1$.
By Lemma~\ref{lem:121121},
the EVs of $P$ are in one-to-one correspondence with the EVs of $\tilde C$.
Hence, there exist finite sets $E_0 = \{ e \mid \eAez{} \in \tilde E_0 \} \subseteq C^r$
and $E_1 = \{ e \mid \eAeo{} \in \tilde E_1 \} \subseteq P$
of EVs such that
\[
x = \sum_{e \in E_0} e + \sum_{e \in E_1} \lambda_e e \quad \text{with } \sign(e) \le \sign(x) .
\]
The set $\tilde E = \tilde E_0 \cup \tilde E_1$ (and hence $E = E_0 \cup E_1$) can be chosen
such that $|E| = |\tilde E| \le \dim(\tilde S) = \dim(P)+1$
and $|E| = |\tilde E| \le |\supp \! \xAxo{}| = |\supp(x)|+1+|\supp(Ax-b)|$.
\end{proof}


\bibliographystyle{plain}
\bibliography{flux_opt}

\begin{thebibliography}{10}

\bibitem{BachemKern1992}
Achim Bachem and Walter Kern.
\newblock {\em Linear programming duality}.
\newblock Springer-Verlag, Berlin, 1992.
\newblock An introduction to oriented matroids.

\bibitem{Caratheodory1911}
Constantin Carath{\'e}odory.
\newblock {\"Uber den Variabilit\"atsbereich der {\it Fourier}schen Konstanten
  von positiven harmonischen Funktionen.}
\newblock {\em {Rend. Circ. Mat. Palermo}}, 32:193--217, 1911.

\bibitem{KlamtMahadevan2015}
Steffen Klamt and Radhakrishnan Mahadevan.
\newblock {{O}n the feasibility of growth-coupled product synthesis in
  microbial strains}.
\newblock {\em Metab. Eng.}, 30:166--178, Jul 2015.

\bibitem{KlamtStelling2003}
Steffen Klamt and J{\"o}rg Stelling.
\newblock {{T}wo approaches for metabolic pathway analysis?}
\newblock {\em Trends Biotechnol.}, 21(2):64--69, Feb 2003.

\bibitem{LlanerasPico2010}
Francisco Llaneras and Jes\'us Pic\'o.
\newblock {{W}hich metabolic pathways generate and characterize the flux space?
  {A} comparison among elementary modes, extreme pathways and minimal
  generators}.
\newblock {\em J. Biomed. Biotechnol.}, 2010:753904, 2010.

\bibitem{Marashi2012}
Sayed-Amir Marashi, Laszlo David, and Alexander Bockmayr.
\newblock {{A}nalysis of metabolic subnetworks by flux cone projection}.
\newblock {\em Algorithms Mol Biol}, 7(1):17, 2012.

\bibitem{Minkowski1896}
Hermann Minkowski.
\newblock {\em {Geometrie der Zahlen (Erste Lieferung)}}.
\newblock Teubner, Leipzig, 1896.

\bibitem{MuellerRegensburgerSteuer2014}
Stefan M{\"u}ller, Georg Regensburger, and Ralf Steuer.
\newblock {E}nzyme allocation problems in kinetic metabolic networks: optimal
  solutions are elementary flux modes.
\newblock {\em J. Theor. Biol.}, 347:182--190, Apr 2014.

\bibitem{Rockafellar1969}
Ralph~T. Rockafellar.
\newblock The elementary vectors of a subspace of {$R^{N}$}.
\newblock In {\em Combinatorial {M}athematics and its {A}pplications ({P}roc.
  {C}onf., {U}niv. {N}orth {C}arolina, {C}hapel {H}ill, {N}.{C}., 1967)}, pages
  104--127. Univ. North Carolina Press, Chapel Hill, N.C., 1969.

\bibitem{Schuster}
Stefan Schuster.
\newblock Personal communication at MPA 2015, Braga, Portugal, 2015.

\bibitem{SchusterHilgetag1994}
Stefan Schuster and Claus Hilgetag.
\newblock On elementary flux modes in biochemical reaction systems at steady
  state.
\newblock {\em J. Biol. Syst.}, 2:165--182, 1994.

\bibitem{SchusterHilgetagWoodsFell2002}
Stefan Schuster, Claus Hilgetag, John~H. Woods, and David~A. Fell.
\newblock Reaction routes in biochemical reaction systems: algebraic
  properties, validated calculation procedure and example from nucleotide
  metabolism.
\newblock {\em J. Math. Biol.}, 45(2):153--181, 2002.

\bibitem{Urbanczik2007}
Robert Urbanczik.
\newblock {{E}numerating constrained elementary flux vectors of metabolic
  networks}.
\newblock {\em IET Syst Biol}, 1(5):274--279, Sep 2007.

\bibitem{UrbanczikWagner2005}
Robert Urbanczik and Clemens Wagner.
\newblock Functional stoichiometric analysis of metabolic networks.
\newblock {\em Bioinformatics}, 21(22):4176--4180, 2005.

\bibitem{Zanghellini2013}
J{\"u}rgen Zanghellini, David~E. Ruckerbauer, Michael Hanscho, and Christian
  Jungreuthmayer.
\newblock Elementary flux modes in a nutshell: Properties, calculation and
  applications.
\newblock {\em Biotechnology Journal}, 8(9):1009--1016, 2013.

\bibitem{Ziegler1995}
G{\"u}nter~M. Ziegler.
\newblock {\em Lectures on polytopes}.
\newblock Springer-Verlag, New York, 1995.

\end{thebibliography}

\end{document}